\documentclass[abstracton]{scrartcl}

\deffootnote{1.5em}{1em}{\textsuperscript{\thefootnotemark}\ }

\usepackage[utf8]{inputenc}
\usepackage[british]{babel}

\usepackage{amsmath,amssymb,amsthm}
\swapnumbers
\theoremstyle{plain}
\newtheorem{theorem}{Theorem}[section]

\newtheorem{obs}[theorem]{Observation}
\newtheorem{prop}[theorem]{Proposition}
\newtheorem{cor}[theorem]{Corollary}
\theoremstyle{definition}
\newtheorem{spltng}[theorem]{Splitting a finite antichain}

\let\notto\nleq
\def\downs#1{{#1}^{\downarrow}}
\def\ups#1{{#1}^{\uparrow}}
\DeclareMathSymbol{\heyt}{\mathbin}{symbols}{"29}
\let\heq\sim
\let\inc\parallel

\def\A{\mathcal{A}}
\def\F{\mathcal{F}}

\def\CD{\mathcal{C}(\Delta)}
\def\D{\mathcal{D}}
\def\P{\mathcal{P}}
\def\S{\mathcal{S}}
\def\DD{\mathbb{D}}
\def\GG{\mathbb{G}}

\def\Ds{$\Delta$-structure}
\def\Df{$\Delta$-forest}
\def\Dt{$\Delta$-tree}

\def\bs#1{\underline{#1}}

\DeclareMathOperator{\Block}{Block}
\DeclareMathOperator{\Inc}{Inc}

\begin{document}

\title{Splitting finite antichains\\in the homomorphism order}

\author{
Jan Foniok\\
	\small ETH Zurich, Institute for Operations Research\\[-5pt]
	\small R\"amistrasse 101, 8092 Zurich, Switzerland\\[-5pt]
	\small E-mail: \texttt{foniok@math.ethz.ch}
\and
Jaroslav Nešetřil\\
	\small Charles University, Department of Applied Mathematics\\[-5pt]
	\small and Institute for Theoretical Computer Science (ITI)\thanks{%
		The Institute for Theoretical Computer Science (ITI) is supported as
		project 1M0545 by the Ministry of Education of the Czech Republic.}\\[-5pt]
	\small Malostranské nám. 25, 118 00 Praha, Czech Republic\\[-5pt]
	\small E-mail: \texttt{nesetril@kam.mff.cuni.cz}
}

\date{Extended Abstract\\\small 9th March 2008}

\maketitle

\begin{abstract}
A structural condition is given for finite maximal antichains in the
homomorphism order of relational structures to have the splitting
property. It turns out that non-splitting antichains appear only at the
bottom of the order. Moreover, we examine looseness and finite antichain
extension property for some subclasses of the homomorphism poset. Finally,
we take a look at cut-points in this order.

\bigskip

\end{abstract}

\section{Introduction}

A homomorphism from a graph~$G$ to a graph~$H$ is a mapping $f:V(G)\to
V(H)$ that preserves the edges of~$G$, that is if $xy\in E(G)$ then
$f(x)f(y)\in E(H)$. The omission of round or curly brackets indicates that the
definition applies to both undirected and directed graphs, as does the
discussion in the next few paragraphs. The definition applies equally
well to finite and infinite graphs, but all graphs we consider in this
paper are finite.

Existence of homomorphisms defines a relation on the class of all
graphs, which is reflexive (the identity mapping is a homomorphism) and
transitive (the composition of homomorphisms is a homomorphism); thus
it is a preorder. Hence we write $G\le H$ if there exists a homomorphism
from~$G$ to~$H$.

Furthermore we write $G\sim H$ if $G\le H$ and at the same time
$H\le G$. Since $\le$~is a preorder, the relation~$\sim$ is an equivalence
relation. The preorder~$\le$ induces a partial order on equivalence
classes of~$\sim$, which is called the \emph{homomorphism order}.

A \emph{core} is a graph which admits a homomorphism to no proper
subgraph of itself. It is easy to show that there is exactly one
core (up to isomorphism) in each equivalence class of~$\sim$
(see~\cite{HelNes:GrH}). So we have a canonical representative
in each class, and sometimes it may be convenient to consider the
homomorphism order to be a relation on the set of all cores rather than
$\sim$-equivalence classes.

Several properties of the homomorphism order have been examined. Perhaps
the earliest result is its universality: Hedrlín~\cite{Hed:Uni} proved
that the homomorphism order contains every countable partial order as
an induced suborder. It came as a surprise that the same is true about
the homomorphism order of finite directed paths, which was proved by
Hubička and Nešetřil~\cite{HubNes:Finite,HuNe:Paths}. Another considered property
was density. Welzl~\cite{Wel:Dense} showed that there is only one gap
in the homomorphism order of undirected graphs. There are infinitely
many gaps in the order of directed graphs, and they were described by
Nešetřil and Tardif~\cite{NesTar:Dual}; we give an overview in
Section~\ref{ss:dual}.

We are interested in finite maximal antichains in the homomorphism order.

As one of the first applications of the probabilistic method,
P.~Erdős~\cite{Erd:Gtp} showed that there exist undirected graphs with
arbitrary large girth and arbitrary large chromatic number. Thus if
$\A$~is a set of undirected graphs that are not bipartite, there exists
a graph~$H$ that is incomparable with all elements of~$\A$: it suffices
to take~$H$ which has both chromatic number and girth larger than any
graph in~$\A$. Hence each finite maximal antichain contains a bipartite
graph or a graph with a loop; but there are only two bipartite cores and
only one core with a loop. Therefore there are only three finite maximal
antichains: $\{K_1\}$, $\{K_2\}$ and $\{\top\}$, where $\top$~denotes
the graph consisting of a single vertex with a loop.

The situation is more intricate in the order of digraphs. There are four
one-element maximal antichains: $\{K_1\}$, $\{\vec P_1\}$, $\{\vec P_2\}$
and $\{\top\}$; where $\top$~is the graph with a single vertex with a
loop as before, and $\vec P_k$~denotes the directed path with $k$~edges.

Maximal antichains of size~2 were classified by Nešetřil and
Tardif~\cite{NesTar:MAC}. They showed that they are the sets $\{F,D\}$
such that $(F,D)$ is a \emph{homomorphism duality}: a pair of graphs
satisfying that $F\le X$ is for any digraph~$X$ equivalent to $X\nleq D$.

In general, a maximal antichain~$\A$ \emph{splits} if it is
a disjoint union $\A=\F\cup\D$ such that (using fairly standard
notation\footnote{$\ups{\F}=\{X: \exists F\in\F,\ F\le X\}$ is the upset
generated by~$\F$; and $\downs{\D}=\{X: \exists D\in\D,\ X\le D\}$ is the
downset generated by~$\D$}) the whole poset $\P=\ups{\F}\cup \downs{\D}$.
Equivalently, $(\F,\D)$ is a splitting of~$\A$ if and only if
$\ups\A=\ups\F$ and $\downs\A=\downs\D$.  The result mentioned in the
previous paragraph implies that all maximal antichains of size~2 in
the homomorphism order of digraphs split, because $\ups\A=\ups F$ and
$\downs\A=\downs D$. Later, Foniok, Nešetřil and Tardif~\cite{FNT:GenDu}
proved that in the homomorphism order of digraphs, all finite maximal
antichains of size greater than one split.

The splitting property of maximal antichains in general posets has been
extensively studied, see~\cite{AhlErdGra:A-splitting,AhlKha:Splitting,
DufSan:Finite,DufSan:Splitting,Dza:A-Note,Erd:Splitting,Erd:Some,ErdSou:How-to-split}.
Here we are concerned with the splitting property of maximal antichains
in the homomorphism order of finite relational structures. Relational
structures are generalisations of graphs; they have a set of vertices
and multiple sets of edges, each of which is a relation that is
not necessarily binary.  We introduce relational structures and the
homomorphism order in Section~\ref{sec:horse}.

Extension of antichains is the topic of
Section~\ref{sec:ext}. There we summarise and extend recent results
of~\cite{DufErdNes:Antichains-in-the-homomorphism}. Several properties
are defined that imply extensibility of finite antichains to splitting
or non-splitting infinite maximal antichains. We also look at the problem
of cut-points, elements of the poset that cut an interval into two. Some
cut-points are again linked to dualities.

In Section~\ref{sec:split}, we show that almost all finite maximal
antichains in this homomorphism order have the splitting property. The
exceptions appear at the very bottom of the order, and are formed
by structures with few edges. For the antichains that split we show
how to partition them into $\F$ and~$\D$. This has been known for
digraphs~\cite{FNT:WG06} and structures with one relation of arbitrary
arity~\cite{FNT:GenDu}.

Lastly, we list some open questions and suggestions for future work
in Section~\ref{sec:conc}.

\section{Homomorphism order of relational structures}
\label{sec:horse}

\subsection{Relational structures}
\label{ss:rs}

A \emph{type}~$\Delta$ is a sequence $(\delta_i: i\in I)$ of
positive integers; $I$~is a finite set of indices. A \emph{relational
structure}~$A$ of type $\Delta$ is a pair $(X, (R_i: i\in I))$, where
$X$~is a finite nonempty set and $R_i\subseteq X^{\delta_i}$; that
is, $R_i$ is a $\delta_i$-ary relation on~$X$. We often refer to a
relational structure of type~$\Delta$ as a \emph{\Ds}. In this paper,
we do not consider unary relations; so we assume that $\delta_i\ge 2$
for all $i\in I$.

If $A=(X,(R_i:i\in I))$, the \emph{base set}~$X$ is denoted
by~$\bs{A}$ and the relation~$R_i$ by~$R_i(A)$. The elements of the base
set are called \emph{vertices} and the elements of the relations~$R_i$
are called \emph{edges}; this terminology is motivated by the fact that
relational structures of type $\Delta=(2)$ are digraphs.  To distinguish
between various relations of a \Ds\ we speak about \emph{kinds of edges}
(so the elements of~$R_i(A)$ are referred to as the \emph{edges of the
$i$th kind}).

Let $A$ and $A'$ be two relational structures of the same type~$\Delta$. A
mapping $f: \bs A\to\bs{A'}$ is a \emph{homomorphism} from~$A$ to~$A'$
if for every $i\in I$ and for every $u_1, u_2, \dotsc, u_{\delta_i}\in
\bs A$ the following implication holds:
\[(u_1, u_2, \dotsc, u_{\delta_i})\in R_i(A) \quad\Rightarrow\quad
(f(u_1), f(u_2), \dotsc, f(u_{\delta_i}))\in R_i(A').\]
An \emph{endomorphism} is a homomorphism from a \Ds\ to itself.

If there exists a homomorphism from~$A$ to~$A'$, we say that $A$~is
\emph{homomorphic} to~$A'$ and write $A\le A'$; otherwise we write
$A\nleq A'$. If $A$~is homomorphic to~$A'$ and at the same time $A'$~is
homomorphic to~$A$, we say that $A$ and $A'$ are \emph{homomorphically
equivalent} and write $A\heq A'$.  If on the other hand there exists
no homomorphism from~$A$ to~$A'$ and no homomorphism from~$A'$ to~$A$,
we say that $A$ and~$A'$ are \emph{incomparable} and write $A\inc A'$.

A \Ds~$C$ is a \emph{core} if it is not homomorphic to any proper
substructure of itself.  Equivalently, $C$~is a core if every endomorphism
of~$C$ is an automorphism. It is well-known (consult~\cite{HelNes:GrH})
that every \Ds~$A$ is homomorphically equivalent up to isomorphism to
exactly one core~$C$; then $C$~is called \emph{the core} of~$A$.
The class of all \Ds s which are cores is denoted by~$\CD$.

\subsection{Homomorphism order}

Let $\Delta$ be a fixed type.  The relation $\le$ of being homomorphic
is reflexive, as the identity mapping is a homomorphism from a \Ds\ to
itself, and it is transitive, since the composition of two homomorphisms
is a homomorphism.  Thus $\le$~is a preorder on the class of all \Ds s.

This preorder induces a partial order on $\heq$-equivalence classes, which
is naturally equivalent to the partial order $\le$ on the class~$\CD$
of all core \Ds s (taken up to isomorphism). This order is called the
\emph{homomorphism order}.

Note that $\CD$~is a distributive lattice: The supremum of two
structures~$A,B$ is their disjoint union~${A+B}$ and the infimum is
the categorical product~${A\times B}$ (a precise definition is found
e.g. in~\cite{FNT:GenDu,HelNes:GrH}).\footnote{Strictly we should say that the
supremum is \emph{the core of} the disjoint union; similarly for
infimum. We allow ourselves the concession to be a little imprecise,
which we expect to make the exposition clearer rather than more confused.}

\subsection{Dualities and gaps}
\label{ss:dual}

In this section we sum up the results of~\cite{FNT:GenDu,NesTar:Dual}. We
give the definition and the characterisation of homomorphism dualities
and describe their connection to gaps in the homomorphism order.

Let $\F$ and~$\D$ be two finite sets of core \Ds s such that no
homomorphisms exist among the structures in~$\F$ and among the structures
in~$\D$. We say that $(\F,\D)$ is a \emph{finite homomorphism duality}
(often just a \emph{finite duality}) if for every \Ds~$A$ there exists
$F\in\F$ such that $F \le A$ if and only if for all $D\in\D$ we have
$A\notto D$. In the special case that $|\F|=|\D|=1$, if $(\{F\},\{D\})$
is a duality pair, that is if $\{A: F\le A\}=\{A: A\notto D\}$, the pair
$(F,D)$ is a \emph{duality pair}.

For the full description of finite dualities we need some more notions.
The \emph{incidence graph}\index{incidence graph}~$\Inc(A)$ of a \Ds~$A$
is the bipartite multigraph
$(V_1\cup V_2, E)$ with parts $V_1=\bs A$ and
\[V_2=\Block(A):=\bigl\{(i,(a_1,\dots,a_{\delta_i})) : i\in I,\
(a_1,\dots,a_{\delta_i})\in R_i(A)\bigr\},\]
and one edge between $a$ and $(i,(a_1,\dots,a_{\delta_i}))$
for each occurrence of~$a$ as some $a_j$ in an edge
$(a_1,\dots,a_{\delta_i})\in R_i(A)$.

A \Ds~$A$ is \emph{connected} if its incidence graph $\Inc(A)$~is
connected; it is a \Dt\ if $\Inc(A)$~is a tree; and it is a \Df\
if $\Inc(A)$ is a \Df. A \emph{component} of a \Ds\ is its maximal
connected substructure. Note that a \Ds\ is a \Df\ if and only if each
of its components is a \Dt.

Let us now give the characterisation of finite dualities.

\begin{theorem}[\cite{FNT:GenDu,NesTar:Dual}]
If $(F,D)$ is a duality pair, then $F$~is a \Dt. Conversely, if $F$~is
a \Dt, there exists a unique \Ds~$D$ (\emph{the dual of~$F$}) such that
$(F,D)$~is a duality pair.

If $(\F,\{D\})$ is a finite duality, then all elements of~$\F$ are \Dt s
and $D$~is the product of their duals. If $(\F,\D)$~is a finite duality,
then all elements of~$\F$ are \Df s and each element of~$\D$ is the
product of duals of some components of elements of~$\F$. In this case,
$\D$~is determined uniquely by~$\F$.
\qed\end{theorem}

In a poset~$\P$, a \emph{gap} is a pair $(B,T)$ of elements of~$\P$
such that $B<T$ and whenever $B\le X\le T$, then $X=B$ or $X=T$. In
the homomorphism order of undirected graphs, there is exactly one gap:
$(K_1,K_2)$. For digraphs and general \Ds s, the gaps are as follows:

\begin{theorem}[\cite{NesTar:Dual}]
A pair $(B,T)$ is a gap in~$\CD$ if and only if there exists a \Dt~$F$
with dual~$D$ such that $T\times D\le B\le D$ and $T=B+F$. If $T$~is
connected, then there exists a \Ds~$B$ such that $(B,T)$ is a gap if
and only if $T$~is a \Dt. Then $B=T\times D$.
\qed\end{theorem}

\section{Antichains, looseness and cut-points}
\label{sec:ext}

In~\cite{DufErdNes:Antichains-in-the-homomorphism}, a number of properties
of the homomorphism order of graphs and digraphs were investigated.

First we introduce some of the notions defined there. Let $(\P,{\le})$
be a poset.  An element $Y\in\P$ is a \emph{cut-point} if there exist
$X,Z\in\P$ such that $X<Y<Z$ and the interval $[X,Z]=[X,Y]\cup[Y,Z]$. A
subset $\A\subseteq\P$ is \emph{cut-free} if there are no $Y\in\A$ and
no $X,Z\in\P$ such that $X<Y<Z$ and $\A\cap[X,Z]=\A\cap([X,Y]\cup[Y,Z])$.

We write $\A\inc X$ to mean that the element~$X$ is incomparable with
each element of~$\A$.

A subset $\A\subseteq\P$ is an \emph{upward loose kernel} in~$\P$ if
for every finite $\S\subseteq\A$ and every $X\in\P\setminus\ups\S$ there
exists $Y\in\A$ such that $X<Y$ and $\A\inc Y$. Analogously $\A$~is an
\emph{downward loose kernel} in~$\P$ if for every finite $\S\subseteq\A$
and every $X\in\P\setminus\downs\S$ there exists $Y\in\A$ such that $X>Y$
and $\S\inc Y$.

A subset $\A\subseteq\P$ \emph{has no finite maximal antichains} in~$\P$
if there is no finite subset $\S\subseteq\A$ that is a maximal antichain
in~$\P$.  Furthermore, $\A$~has the \emph{finite antichain extension
property} in~$\P$ if for every finite antichain $\S\subseteq\A$ and
every $X\in\P\setminus\S$ there exists $Y\in\A$ such that $\S\inc Y$
but $Y$~is comparable with~$X$. Observe that if $\A$ is both an upward
and a downward loose kernel in~$\P$, then $\A$~has the finite antichain
extension property in~$\P$.

In short, the importance of these notions is this: Finite antichains
in subsets with the finite antichain extension property extend to
(infinite) splitting maximal antichains. And non-maximal antichains in
loose kernels extend to (infinite) non-splitting maximal antichains. For
details see~\cite{DufErdNes:Antichains-in-the-homomorphism}.

Now let $\GG$ be the homomorphism poset of undirected graphs and let
$\GG'=\GG\setminus\{K_1,K_2\}$ be the set of non-bipartite cores. Then
we have:

\begin{theorem}[\cite{DufErdNes:Antichains-in-the-homomorphism}]
The subset $\GG'$ is both an upward loose kernel and a downward loose
kernel in~$\GG$. Hence $\GG'$~has the finite antichain extension property
in~$\GG$ and it has no finite maximal antichains in~$\GG$. Moreover,
$\GG'$~is cut-free in~$\GG$.
\qed\end{theorem}

In addition, $\GG'$ is dense, that is it contains no gaps.

As usual, the situation is more complex for digraphs or \Ds s. To
begin with, it is not entirely clear what subset should play the role
of~$\GG'$. Let us have a look at some conditions, which are equivalent
for undirected graphs.

\begin{obs}
\label{obs:cyc}
Consider $\GG$, the homomorphism order of undirected graphs.
The subset $\GG'$ of\/ $\GG$ is each of the following:\par
(1) the set of all cores that are not homomorphic to any tree;\par
(2) the set of all cores that have no component homomorphic to a tree;\par
(3) the set of all cores that contain a cycle;\par
(4) the set of all cores that contain a cycle in each connected component;\par
(5) the set of all cores that contain an odd cycle;\par
(6) the set of all cores that contain an odd cycle in each connected component.
\qed\end{obs}

In the homomorphism order of digraphs~$\DD$, though, the situation
is contrasting: no two of the sets defined by the conditions (1)--(6)
coincide. This fact motivates separate study of these subsets.

Hence we define subsets $\DD_1,\DD_2,\dotsc,\DD_6$ of~$\DD$ so
that we set $\DD_k$ to be the set satisfying condition~($k$) of
Observation~\ref{obs:cyc}. For instance, $\DD_1$~is the set of all core
digraphs that are not homomorphic to an orientation of a tree.

Analogously, in the homomorphism order of \Ds s, for $k=1,2,3,4$,
let $\CD_k$ be the subset of~$\CD$ satisfying condition~($k$) of
Observation~\ref{obs:cyc}.

Some properties of subsets of~$\DD$ defined in this way were examined
in~\cite{DufErdNes:Antichains-in-the-homomorphism}.

\begin{theorem}[\cite{DufErdNes:Antichains-in-the-homomorphism}]
\label{thm:dd56}
The subset $\DD_5$ is a downward loose kernel in~$\DD$. The subset
$\DD_6$~is an upward loose kernel in~$\DD$; $\DD_6$ has no finite
maximal antichains; but $\DD_6$~does not have the finite antichain
extension property.
\qed\end{theorem}

The main tool for proving this theorem is \emph{sparse
incomparability}. This essentially guarantees the existence of digraphs
(or \Ds s) that are locally trees (they have large girth) but do not admit
homomorphisms to and from some prescribed graphs. Using similar methods
as in~\cite{DufErdNes:Antichains-in-the-homomorphism}, and exploiting
the characterisation of finite maximal antichains for digraphs, given
in~\cite{FNT:WG06}, we can show the following.

\begin{theorem}
\label{thm:cd2dd34}
Neither $\DD_3$ nor $\DD_4$ has finite maximal antichains in~$\DD$. The
subsets $\CD_1$ and $\CD_2$ have no finite maximal antichains. The
subset~$\CD_1$ is a downward loose kernel in~$\CD$, and $\CD_2$~is an
upward loose kernel in~$\CD$; but neither $\CD_1$ nor $\CD_2$ has the
finite antichain extension property.
\qed\end{theorem}

Cut-points are also related to the splitting property of antichains. It
follows from~\cite[Theorem~2.1]{AhlErdGra:A-splitting} as well
as~\cite[Theorem~2.10]{ErdSou:How-to-split} that a finite maximal
antichain splits if it contains no cut-point. Let us look at cut-points
a little closer.

\begin{prop}
\label{prop:cutp}
Let $T$ be a \Dt\ and let $D$ be its dual. Then the \Ds s $T+D$
and~$T\times D$ are cut-points in the homomorphism order~$\CD$.
\end{prop}

\begin{proof}
Consider the interval~$[K_1, T]$, which is equal to the downset generated
by~$T$. Suppose that $X$~is a \Ds\ such that $X<T$. Then $X\le D$,
because $T\nleq X$. Thus $X\le T\times D$. Hence the interval~$[T\times
D, T]$ contains only its end-points, that is $[T\times D, T]=\{T\times
D, T\}$. Moreover, $[K_1, T\times D]\cup[T\times D, T]=[K_1, T]$,
so $T\times D$~is a cut-point.

Similarly, if $D<X$, then $T+D\le X$. Hence
$[D,T+D]\cup[T+D,\top]=[D,\top]$ and so $T+D$~is a cut-point.
\end{proof}

No cut-free subset can contain cut-points; and there are cycles and odd
cycles in many duals.  Thus $T+D$ is a cut-point belonging to several
of the classes defined above, and hence we have:

\begin{cor}
None of the classes $\CD_1$, $\CD_3$ and $\DD_5$ is cut-free.
\qed\end{cor}

\section{Splitting antichains}
\label{sec:split}

In this section, we prove that many finite maximal antichains in~$\CD$
split. To do so, we construct a partition $(\F,\D)$ for any finite
antichain~$\A$ and show that this partition is often a splitting of~$\A$;
that is, in many cases $\ups\A=\ups\F$ and $\downs\A=\downs\D$. So let
us reveal our construction of the partition.

\begin{spltng}
\label{pgf:macfd}
Let $\A=\{A_1,A_2,\dotsc,A_n\}$ be a finite maximal antichain in~$\CD$. Recursively, define the
sets $\F_0$, $\F_1$,~\dots,~$\F_n$ in this way:
\begin{enumerate}
\item Let $\F_0=\emptyset$.
\item For $i=1,2,\dotsc,n$: check whether there exists a \Ds~$X$ satisfying
\begin{itemize}
\item[(i)] $A_i < X$,
\item[(ii)] $F\notto X$ for any $F\in\F_{i-1}$, and
\item[(iii)] $A_j\notto X$ for any $j>i$.
\end{itemize}
If such a structure~$X$ exists, let $\F_i=\F_{i-1}\cup\{A_i\}$, otherwise let $\F_i=\F_{i-1}$.
\item Finally, let $\F=\F_n$ and $\D=\A\setminus\F$.
\end{enumerate}
\end{spltng}

We are just about to prove that $(\F,\D)$ defined above is a splitting
of~$\A$ unless $\A$~lies ``at the bottom'' of the order. To formalise
``the bottom'', consider $D^\ast$: the \Ds, whose components are exactly
all trees with at most one edge of each kind. So $T$~is a component
of~$D^\ast$ if and only if $T$~is a \Dt\ and $|R_i(T)|\le1$ for all
$i\in I$.

A \Ds~$X$ is \emph{small} if either $X\le D^\ast$ or there exists
$Y\le D^\ast$ such that $Y<X$ and whenever $Y<Z<X$ then $Z\le D^\ast$.

\begin{theorem}
Let $\A$ be a finite maximal antichain in the homomorphism
poset~$\CD$. Let $\F$, $\D$ be defined as in~\ref{pgf:macfd}. If no
element of~$\F$ is small, then $(\F,\D)$ is a splitting of~$\A$, that
is $\ups\A=\ups\F$ and $\downs\A=\downs\D$.
\end{theorem}

\begin{proof}[Idea of proof]
The construction in~\ref{pgf:macfd} ensures that $\ups\A=\ups\F$. So it
remains to prove that $\downs\A=\downs\D$. We will assume that there is
a \Ds~$Y\in\downs\A\setminus\downs\D$ and prove that then $\F$~contains
a small element.

By definition of~$Y$, there exists $F\in\F$ such that $Y\le F$.  Using a
variant of sparse incomparability, it can be shown that each element
of~$\F$ is homomorphic to a \Dt\ (it is \emph{balanced}). Therefore
$Y$~is also balanced.

Next comes a cycle-growing trick. We consider forbidden paths (whose
precise definition is technical and we omit it here) and construct big
unbalanced structures~$W$ from them. Some considerations show  that for
each such $W$ there is $F\in\F$ such that $F\le W+Y$. Then it follows
that $F\le P+Y$, where $P$~is the appropriate forbidden path. But since
$F\notto Y$, we have $P\notto Y$.

Then we prove that a connected \Ds, to which no forbidden path is
homomorphic, admits a homomorphism to a \Dt\ with at most one edge of
each kind. Hence $Y\le D^\ast$ whenever $Y\in\downs\A\setminus\downs\D$.

So we have $Y\le F$ for some $F\in\F$. If $F\to D^\ast$, then $F$~is small
and we are done.  Otherwise suppose $Y\le X\le F$. Then each $X$ such
that $Y<X<F$ satisfies that $X\le D^\ast$, therefore $F$~is again small.
\end{proof}

In conclusion, we remark that splitting finite antichains are homomorphism
dualities in disguise. Indeed, the splitting $(\F,\D)$ of a finite
maximal antichain is a homomorphism duality. Conversely, let $(\F,\D)$
be a homomorphism duality and let
\[\A=\F\cup\{D\in\D: D\notto F\text{ for any } F\in\F\}.\]
Then $\A$~is obviously a finite maximal antichain in~$\CD$. For relational
structures with one or two relations all maximal antichains, even the
non-splitting ones, are created from finite dualities in this way. For
structures with more than two relations this is currently unknown.

\section{Conclusion and extensions}
\label{sec:conc}

Theorems \ref{thm:dd56} and \ref{thm:cd2dd34} do not cover all the
properties for all the classes $\DD_k$, $\CD_k$. We would like to find
out which classes have which properties.  In particular, the question is
interesting for structures that contain cycles only in some components,
and for structures with balanced cycles.

Another class of digraphs is introduced
in~\cite{DufErdNes:Antichains-in-the-homomorphism}: digraphs that contain
a directed cycle. This subset has the finite antichain extension property
and is cut-free. This class can be generalised to \Ds s by looking at
their \emph{directed shadows} (these are directed multigraphs constructed
by replacing each tuple $(x_1,\dotsc, x_t)$ of the \Ds\ with the directed
path $x_1 \to x_2 \to \cdots \to x_t$).  This class is an upward loose
kernel in~$\CD$ and it has the finite antichain extension property. Is
it also a downward loose kernel in~$\CD$?

In Proposition~\ref{prop:cutp} we describe infinitely many cut-points in
the homomorphism order.  Currently we do not know any other cut-points;
however, it remains open to give the complete characterisation.
In particular, can such a characterisation yield a new proof of the
splitting property (recall that a finite maximal antichain splits if it
contains no cut-points)?

Finally, the connection between gaps and dualities is not
restricted to the homomorphism order. A similar theory was developed
in~\cite{NesPulTar:HeytDual} for Heyting lattices. Thus it is worth
asking what additional axioms enable deriving simple conditions for
maximal antichains to split.

\end{document}